\documentclass[12pt]{amsart}

\usepackage{amssymb,amscd,amsthm, verbatim,amsmath,color,fancyhdr, mathrsfs}
\usepackage{graphicx}
\usepackage{turnstile}

\usepackage[letterpaper, left=2.5cm, right=2.5cm, top=2.5cm,
bottom=2.5cm,dvips]{geometry}

\newtheorem{thm}{Theorem}[section]

\newtheorem{lem}[thm]{Lemma}
\newtheorem{rem}[thm]{Remark}
\newtheorem{cor}[thm]{Corollary}

\newtheorem{ex}[thm]{Example}
\newtheorem{defn}[thm]{Definition}

\numberwithin{equation}{section}

\title[A higher-order evolution equation with an inhomogeneous term]{Nonexistence results for a higher-order evolution equation with an inhomogeneous term depending on time and space}

\author[M. Jleli, N.-A. Lai, B. Samet]{Mohamed Jleli, Ning-An Lai, Bessem Samet}



\begin{document}

\maketitle

\begin{abstract}
We consider  a higher-order evolution equation with an inhomogeneous term depending on time and space.  We first derive a general criterion for the nonexistence of weak solutions. Next, we study the particular case when the inhomogeneity depends only on space. In that case, we obtain the first critical exponent in the sense of Fujita, as well as  the second critical exponent in the sense of Lee and Ni.
\end{abstract}

\section{Introduction}

In this paper, we  investigate  the questions of  existence and nonexistence of weak solutions to the inhomogeneous problem
\begin{equation}\label{P}
\square_k u = |u|^p + |\partial^{k-1}_{t} u|^q+ w(t,x)\quad \mbox{in}\quad (0,\infty)\times \mathbb{R}^N.
\end{equation}
Here $\square_k:=\partial^k_{t}-\Delta$ ($k\geq 2$), $\partial^i_{t}:=\frac{\partial}{\partial t^i}$, $p,q>1$, $N\geq 1$ and  $w\geq 0$ is a nontrivial $L^1_{loc}$ function. We mention below some motivations for studying problems of type \eqref{P}.

In the case $k=2$ and $w\equiv 0$, problem \eqref{P} reduces to
\begin{equation}\label{PHZ}
\square_2= |u|^p +|\partial_t u|^q\quad \mbox{in}\quad (0,\infty)\times \mathbb{R}^N.
\end{equation}
This problem can be considered as a natural combination of the problem
\begin{equation}\label{PH}
\square_2 u  = |u|^p\quad \mbox{in}\quad (0,\infty)\times \mathbb{R}^N
\end{equation}
and the problem
\begin{equation}\label{PG}
\square_2 u = |\partial_t u|^q\quad \mbox{in}\quad (0,\infty)\times \mathbb{R}^N.
\end{equation}
John \cite{J} proved that  problem \eqref{PH} in $\mathbb{R}^3$ admits as critical exponent $p_c(3):=1+\sqrt{2}$, in the sense that: when $1< p < p_c(3)$ the solution blows up in a finite time, while for $p>p_c(3)$, there exists global solution. Strauss \cite{S}  conjectured that for all $N\geq 2$, problem \eqref{PH} admits as critical exponent the real number $p_c(N)$, which is the positive root of the equation
$$
(N-1)p^2-(N+1)p-2=0,
$$
that is,
$$
p_c(N)=\frac{(N+1)+\sqrt{N^2+10N-7}}{2(N-1)}.
$$
Later, several  mathematicians have contributed to solve this conjecture, see e.g. \cite{GE,G1,G2,JI,LA,L,RA,SC,SI,TW,Y,ZH,ZH2}. For problem \eqref{PG}, it was conjectured (Glassey conjecture) that the critical exponent is given by
$$
q_c(N):=1+\frac{2}{N-1}.
$$
In the case $q\leq q_c(N)$, the nonexistence of  global small solutions  for problem \eqref{PG} was investigated by several authors, see e.g.
\cite{A,J2,RAA,SC2,ZHH}. In the case $q>q_c(N)$, the existence of global small solutions has been established in \cite{HT,SI2,T} for $N\in\{2,3\}$, and    in \cite{HWY} for $N\geq 4$, under the radial assumption of the initial data. In \cite{HZ}, Han and Zhou investigated the blow-up phenomenon for problem \eqref{PHZ}. Namely, they obtained blow-up results when $p>p_c(N)$, $q>q_c(N)$ and $(p-1)((N-1)q-2)< 4$. Moreover, in certain cases, they obtained an upper bound of the lifespan. In \cite{HWY2}, Hidano,  Wang and Yokoyama investigated problem \eqref{PHZ} in the case $N\in\{2,3\}$. Namely, they determined the full region of $(p,q)$ for which  there is  global existence of small solutions. Moreover, a sharp lower bound of the lifespan was obtained for many $(p,q)$ when there is no global existence.

In \cite{Zhang}, Zhang studied the inhomogeneous semilinear wave equation
\begin{equation}\label{PZhang}
\square_2 u = |u|^p +  w(x)\quad \mbox{in}\quad (0,\infty)\times \mathbb{R}^N,
\end{equation}
where $w\geq 0$ is a nontrivial $L^1_{loc}$ function. He proved that, if one assumes that stationary solutions are global solutions, then the critical exponent for problem \eqref{PZhang} is $p^*=\frac{N}{N-2}$, $N\geq 3$. Namely, he showed that, when $1<p<p^*$, then problem \eqref{PZhang} possesses no global solutions for any initial values; when $p>p^*$, then
problem \eqref{PZhang}  has  global solutions (more precisely, stationary solutions) for some $w>0$ and suitable initial values.  Note that, in the blow-up case,  no assumptions on the sign or decay of the initial values were supposed.

In \cite{MP}, Mitidieri and Pohozaev considered the inhomogeneous exterior problem

\begin{eqnarray}\label{PMP}
\left\{\begin{array}{lllll}
\square_k u &\geq & |u|^p +  w(x) &\mbox{in}&  (0,\infty)\times \Omega,\\
u &\geq & 0 &\mbox{on}&  (0,\infty)\times \partial \Omega,\\
\partial^{k-1}_t u(0,x) &\geq &0  &\mbox{in}&   \Omega,
\end{array}
\right.
\end{eqnarray}
where $\Omega$ is the exterior of a ball of center $0$ and radius $R>0$, and $w\geq 0$ is a nontrivial $L^1_{loc}$ function. Using the test function method, they proved that, if $1<p<p^*$, then problem \eqref{PMP} admits no weak solutions.

As far as we know, problems of type \eqref{P}  were not considered previously in the literature. Before stating  the main results related to problem \eqref{P}, let us mention in which sense solutions  are considered. Just before, let us fix some notations.  Denote
$$
Q=(0,\infty)\times\mathbb{R}^N.
$$
By $C_c^2(Q)$, we mean  the  space of  $C^2$ real valued functions compactly supported in  $Q$.  Let
\begin{equation}\label{pstar}
p^*(N)=\left\{\begin{array}{lll}
\infty &\mbox{if}& N\in\{1,2\},\\
\frac{N}{N-2}&\mbox{if}& N\geq 3.
\end{array}
\right.
\end{equation}

\begin{defn}
We say that $(u,\partial^{k-1}_{t} u)\in L^p_{loc}(Q)\times L^q_{loc}(Q)$ is a  weak solution to problem \eqref{P}, if for any $\varphi\in C_c^2(Q)$, there holds
\begin{equation}\label{ws}
\int_Q \left(|u|^p\varphi+|\partial^{k-1}_{t} u|^q+w(t,x)\right)\varphi\,dx\,dt
=-\int_Q \partial^{k-1}_t u\,\partial_t \varphi\,dx\,dt-\int_Q u \Delta \varphi\,dx\,dt.
\end{equation}
\end{defn}

Our first main theorem provides a general nonexistence result for problem \eqref{P}.

\begin{thm}\label{T1}
Let $w\in L^1_{loc}(Q)$, $w\geq 0$.
Suppose that there exist $0<c_1<c_2<1$  such that
\begin{equation}\label{gbl}
\limsup_{T\to \infty} T^{\frac{q}{q-1}\left[1-\frac{N(p-1)}{2p}\right]-1} \int_{c_1T}^{c_2T} \int_{0< |x|< T^{\frac{q(p-1)}{2p(q-1)}}} w(t,x)\,dx\,dt =+\infty.
\end{equation}
Then problem \eqref{P} admits no  weak solutions.
\end{thm}

\begin{rem}
{\rm{(i)}} Note  that  no conditions on the sign or decay of the initial values
$\partial^i_t u(0,x)$, $i=0,1,\cdots,k-1$, are imposed in Theorem \ref{T1}. \\
\smallskip
{\rm{(ii)}} Observe that condition \eqref{gbl} is independent of $k$.
\end{rem}

In the particular case $w=f(t)g(x)$, one deduces from Theorem \ref{T1} the following nonexistence result.

\begin{cor}\label{cor1}
Let $w=f(t)g(x)$, where $f\in L^1_{loc}((0,\infty))$, $g\in L^1_{loc}(\mathbb{R}^N)$, $f,g\geq 0$ and $g\not\equiv 0$. Suppose that  there exist $0<c_1<c_2<1$ such that
\begin{equation}\label{blu}
\limsup_{T\to \infty} T^{\frac{q}{q-1}\left[1-\frac{N(p-1)}{2p}\right]-1}  \int_{c_1T}^{c_2T} f(t)\,dt =+\infty.
\end{equation}
Then problem \eqref{P} admits no  weak solutions.
\end{cor}

\begin{ex}
Let $f\in L^1_{loc}((0,\infty))$, $f\geq 0$,  be a  function satisfying, for some $C>0$ and $-1+\frac{1}{q}<\sigma<-1+\frac{N}{2}$,
$$
f(t)\geq C t^{\frac{q\sigma}{q-1}},
$$
for $t$  sufficiently large. Let $w=f(t)g(x)$, where $g\in L^1_{loc}(\mathbb{R}^N)$,  $g\geq 0$ and $g\not\equiv 0$. For $T$ sufficiently large, one has
$$
\int_{\frac{T}{2}}^T f(t)\,dt\geq C(\sigma,q) T^{1+\frac{q\sigma}{q-1}},
$$
where $C(\sigma,q)>0$ is a constant that depends only of $\sigma$ and $q$.
Hence, one obtains
$$
T^{\frac{q}{q-1}\left[1-\frac{N(p-1)}{2p}\right]-1}  \int_{\frac{T}{2}}^T f(t)\,dt\geq C(\sigma,q) T^{\frac{q}{q-1}\left[\sigma+1-\frac{N(p-1)}{2p}\right]}.
$$
Therefore, by Corollary \ref{cor1}, one deduces that, if
$$
1<p<\frac{N}{N-2(\sigma+1)}\quad\mbox{and}\quad  q>\max\left\{1,\frac{2}{N}\right\},
$$
then problem \eqref{P} admits no  weak solutions.
\end{ex}

In the case $w=g(x)$, the next theorem provides the critical exponent  for problem \eqref{P} in the sense of Fujita \cite{Fujita}.

\begin{thm}[First critical exponent]\label{T2}
{\rm{(I)}} Let $w=g(x)$,  $g\in L^1_{loc}(\mathbb{R}^N)$, $g\geq 0$, $g\not\equiv 0$, $q>1$ and  $1<p<p^*(N)$, where $p^*(N)$ is given by \eqref{pstar}. Then problem \eqref{P} admits no weak solutions.\\
\smallskip
{\rm{(II)}} If $N\geq 3$ and $p>p^*(N)$, then for all $q>1$, problem \eqref{P} has global positive solutions for some $w=g(x)>0$ and  suitable initial values.
\end{thm}

\begin{rem}
{\rm{(i)}} In this paper, we are concerned essentially with blow-up results. Comparing with the delicate contributions on global existence on the homogeneous case, the existence result (II) in Theorem \ref{T2} (same remark for the existence result (II) in Theorem \ref{T3}) is just a consequence of  elliptic results. We hope this can be improved in a future work.\\
\medskip
{\rm{(ii)}}  If one excludes stationary solutions as global solutions, then it is not certain that the exponent  $p^*(N)$  is still critical.\\
\medskip
{\rm{(iii)}} We do not know whether the exponent $p^*(N)$  belongs to the blow-up case or not.
\end{rem}

Further, for  $a<N$, we define the sets

$$
\mathbb{I}_a=\left\{g\in C(\mathbb{R}^N)\big| g(x)\geq 0,\, g(x)\geq C |x|^{-a} \mbox{ for }|x|\mbox{ large}\right\}
$$
and
$$
\mathbb{J}_a=\left\{g\in C(\mathbb{R}^N)\big| g(x)>0,\, g(x)\leq C |x|^{-a} \mbox{ for }|x|\mbox{ large}\right\},
$$
where $C>0$ is a  constant (independent of $x$).

Denote
$$
a^*=\frac{2p}{p-1}.
$$
The next result provides the second critical exponent for problem \eqref{P}
in the sense of Lee and Ni \cite{LN}.

\begin{thm}[Second critical exponent]\label{T3}
Let $N\geq 3$, $q>1$ and $p>p^*(N)$.\\
\medskip
{\rm{(I)}} If  $a<a^*$ and $w=g(x)\in \mathbb{I}_a$, then  problem \eqref{P} admits no  weak solutions.\\
\smallskip
{\rm{(II)}} If $a^*\leq a<N$, then  problem \eqref{P} has    positive global solutions for some $w=g(x)\in \mathbb{J}_a$ and  suitable initial values.
\end{thm}

The rest of the paper is organized as follows. In Section \ref{sec2}, we provide some preliminary estimates that will be used later in the proofs of our main results. In Section \ref{sec3}, we prove Theorem \ref{T1} and Corollary \ref{cor1}.  In Section \ref{sec4}, we prove Theorems \ref{T2} and \ref{T3}.

\section{Preliminary estimates}\label{sec2}

Given $0<c_1<c_2<1$, let $\eta,\xi \in C^\infty([0,\infty))$ be two functions satisfying
$$
\eta\geq 0, \quad  \mbox{supp}(\eta) \subset (0,1),\quad \eta(t)=1,\, c_1\leq t\leq c_2
$$
and
$$
0\leq \xi\leq 1,\quad
\xi(\sigma)=\begin{cases}
1 &\text{if } 0\leq \sigma\leq 1,\\
0 &\text{if } \sigma \geq 2.
\end{cases}
$$
For $T>0$, let
$$
\phi_T(t,x)=\eta\left(\frac{t}{T}\right)^\ell \xi\left(\frac{|x|^2}{T^{2\theta}}\right)^\ell:=\lambda_T(t)\mu_T(x),\quad (t,x)\in (0,\infty)\times \mathbb{R}^N,
$$
where $\ell,\theta>0$ are to be chosen.

\begin{lem}\label{L1}
Let $m>1$ and $\ell\geq \frac{m}{m-1}$. There holds
$$
\int_{Q} \phi_T^{\frac{-1}{m-1}} |\partial_t \phi_T|^{\frac{m}{m-1}}\,dx\,dt\leq C T^{1+N\theta -\frac{m}{m-1}},\quad T>0.
$$
Here and below, $C$ is a positive constant (independent of $T$), whose value may change from line to line.
\end{lem}

\begin{proof}
Let $T>0$. By the definition of the function $\phi_T$, one has
\begin{equation}\label{es1}
\int_{Q} \phi_T^{\frac{-1}{m-1}} |\partial_t \phi_T|^{\frac{m}{m-1}}\,dx\,dt=\left(\int_0^T \lambda_T^{\frac{-1}{m-1}}
|\lambda_T'|^{\frac{m}{m-1}}\,dt\right) \left(\int_{\mathbb{R}^N} \mu_T(x)\,dx\right).
\end{equation}
Using the change of variable $x=T^\theta y$, one obtains
\begin{equation}\label{es2}
\int_{\mathbb{R}^N} \mu_T(x)\,dx=T^{N\theta}
\int_{0\leq |y|^2\leq 2} \xi(|y|^2)^\ell\,dy=C T^{N\theta}.
\end{equation}
By the change of variable $t=T s$, one gets
\begin{eqnarray*}
\int_0^T \lambda_T^{\frac{-1}{m-1}}
|\lambda_T'|^{\frac{m}{m-1}}\,dt&=&C T^{\frac{-m}{m-1}} \int_0^T \eta\left(\frac{t}{T}\right)^{\ell-\frac{m}{m-1}} \eta'\left(\frac{t}{T}\right)^{\frac{m}{m-1}}\,dt\\
&=&C T^{1-\frac{m}{m-1}}\int_0^1 \eta(s)^{\ell-\frac{m}{m-1}} \eta'(s)^{\frac{m}{m-1}}\,ds,
\end{eqnarray*}
hence
\begin{equation}\label{es3}
\int_0^T \lambda_T^{\frac{-1}{m-1}}
|\lambda_T'|^{\frac{m}{m-1}}\,dt=C T^{1-\frac{m}{m-1}}.
\end{equation}
Finally, \eqref{es1}--\eqref{es3} yield the desired estimate.
\end{proof}

\begin{lem}\label{L3}
Let $m>1$ and $\ell\geq \frac{2m}{m-1}$. There holds
$$
\int_{Q} \phi_T^{\frac{-1}{m-1}} |\Delta \phi_T|^{\frac{m}{m-1}}\,dx\,dt\leq C T^{1+N\theta-\frac{2m\theta}{m-1}},\quad T>0.
$$
\end{lem}

\begin{proof}
Let $T>0$. There holds
\begin{equation}\label{esss1}
\int_{Q} \phi_T^{\frac{-1}{m-1}} |\Delta \phi_T|^{\frac{m}{m-1}}\,dx\,dt=\left(\int_0^T \lambda_T\,dt\right) \left(\int_{\mathbb{R}^N} \mu_T^{\frac{-1}{m-1}} |\Delta \mu_T|^{\frac{m}{m-1}}\,dx\right).
\end{equation}
Using the change of variable $t=T s$, one obtains
\begin{equation}\label{esss2}
\int_0^T \lambda_T\,dt=T \int_0^1 \eta(s)^\ell \,ds =CT.
\end{equation}
By the change of variable $x=T^\theta y$, one has
\begin{eqnarray*}
&\int_{\mathbb{R}^N} \mu_T^{\frac{-1}{m-1}} |\Delta \mu_T|^{\frac{m}{m-1}}\,dx& \\
&\leq  C T^{N \theta-\frac{2m\theta}{m-1}}\left(\int_{1\leq |y|^2\leq 2}\xi(|y|^2)^{l-\frac{2m}{m-1}}|y|^{\frac{2m}{m-1}}dy+\int_{1\leq |y|^2\leq 2}\xi(|y|^2)^{l-\frac{m}{m-1}}dy\right),
\end{eqnarray*}
hence
\begin{equation}\label{esss3}
\int_{\mathbb{R}^N} \mu_T^{\frac{-1}{m-1}} |\Delta \mu_T|^{\frac{m}{m-1}}\,dx \leq  C T^{N\theta-\frac{2m\theta}{m-1}}.
\end{equation}
Finally, \eqref{esss1}--\eqref{esss3} yield the desired estimate.
\end{proof}

\section{A general nonexistence result}\label{sec3}

In this section, we prove the general nonexistence result given by Theorem \ref{T1}, as well as Corollary \ref{cor1}.

\begin{proof}[Proof of Theorem \ref{T1}]
Suppose that $(u,\partial^{k-1}_{t} u)\in L^p_{loc}(Q)\times L^q_{loc}(Q)$  is a  weak solution to problem \eqref{P}. For $T>0$, taking $\varphi=\phi_T$ in \eqref{ws}, one obtains
\begin{equation}\label{wsss}
\begin{aligned}
& \int_Q |u|^p\phi_T\,dx\,dt +\int_Q  |\partial^{k-1}_{t} u|^q\phi_T\,dx\,dt + \int_Q w(t,x) \phi_T\,dx\,dt\\
&\leq  \int_Q |\partial^{k-1}_t u|\,|\partial_t \phi_T|\,dx\,dt+\int_Q |u| |\Delta \phi_T|\,dx\,dt.
\end{aligned}
\end{equation}
Using the $\varepsilon$-Young inequality with $\varepsilon=1$, one gets
\begin{equation}\label{AA1}
\int_{Q} |u||\Delta \phi_T|\,dx\,dt \leq  \int_{Q} |u|^p\phi_T\,dx\,dt+C \int_{Q} \phi_T^{\frac{-1}{p-1}} |\Delta \phi_T|^{\frac{p}{p-1}}\,dx\,dt.
\end{equation}
Similarly, one has
\begin{equation}\label{AA2}
\int_Q |\partial^{k-1}_t u|\,|\partial_t \phi_T|\,dx\,dt \leq  \int_{Q} |\partial^{k-1}_t u|^q\phi_T\,dx\,dt+C \int_{Q} \phi_T^{\frac{-1}{q-1}} |\partial_t \phi_T|^{\frac{q}{q-1}}\,dx\,dt.
\end{equation}
It follows from \eqref{wsss}--\eqref{AA2} that
$$
\int_{Q} w(t,x)\phi_T\,dx\,dt\leq C\left(I_1(T)+I_2(T)\right),
$$
where
$$
I_1(T)=\int_{Q} \phi_T^{\frac{-1}{p-1}} |\Delta \phi_T|^{\frac{p}{p-1}}\,dx\,dt
$$
and
$$
I_2(T)=\int_{Q} \phi_T^{\frac{-1}{q-1}} |\partial_t \phi_T|^{\frac{q}{q-1}}\,dx\,dt.
$$
Taking $\ell= \max\left\{\frac{q}{q-1}, \frac{2p}{p-1}\right\}$ and using Lemma \ref{L3} with $m=p$, one obtains
$$
I_1(T)\leq CT^{1+N\theta-\frac{2p\theta}{p-1}}.
$$
Using Lemma \ref{L1} with $m=q$, one obtains
$$
I_2(T) \leq C T^{1+N\theta -\frac{q}{q-1}}.
$$
Hence, we deduce that
$$
\int_{Q} w(t,x)\phi_T\,dx\,dt\leq C\left(T^{1+N\theta-\frac{2p\theta}{p-1}}+T^{1+N\theta -\frac{q}{q-1}}\right).
$$
Taking $\theta=\frac{(p-1)q}{2(q-1)p}$, one has
$$
1+N\theta-\frac{2p\theta}{p-1}=1+N\theta -\frac{q}{q-1}=1+\frac{q}{q-1}\left[\frac{N(p-1)}{2p}-1\right]
$$
and
\begin{equation}\label{MB}
\int_{Q} w(t,x)\phi_T\,dx\,dt\leq C T^{1+\frac{q}{q-1}\left[\frac{N(p-1)}{2p}-1\right]}.
\end{equation}
On the other hand, by the definition of the function $\phi_T$,  since $w\geq 0$, one gets
\begin{equation}\label{attention}
\int_{Q} w(t,x)\phi_T\,dx\,dt\geq \int_{c_1T}^{c_2T}\int_{0<|x|< T^{\frac{(p-1)q}{2(q-1)p}}} w(t,x)\,dx\,dt.
\end{equation}
It follows from \eqref{MB} and \eqref{attention} that
$$
T^{\frac{q}{q-1}\left[1-\frac{N(p-1)}{2p}\right]-1} \int_{c_1T}^{c_2T}\int_{0<|x|< T^{\frac{(p-1)q}{2(q-1)p}}} w(t,x)\,dx\,dt \leq C,
$$
which contradicts \eqref{gbl}. This proves Theorem \ref{T1}.
\end{proof}

\begin{proof}[Proof of Corollary \ref{cor1}]
Let $T>0$ be large enough. One has
\begin{eqnarray*}
\int_{c_1T}^{c_2T} \int_{0< |x|< T^{\frac{q(p-1)}{2p(q-1)}}} w(t,x)\,dx\,dt&=&
\left(\int_{c_1T}^{c_2T} f(t)\,dt\right)\left(\int_{0< |x|< T^{\frac{q(p-1)}{2p(q-1)}}} g(x)\,dx\right)\\
&\geq & \left(\int_{0< |x|< 1} g(x)\,dx\right) \left(\int_{c_1T}^{c_2T} f(t)\,dt\right)\\
&=& C \int_{c_1T}^{c_2T} f(t)\,dt.
\end{eqnarray*}
Hence, using \eqref{blu}, one deduces that \eqref{gbl} is satisfied. Therefore, the result follows from Theorem \ref{T1}.
\end{proof}

\section{First and second critical exponents}\label{sec4}

We first prove the Fujita-type result given by Theorem \ref{T2}.

\begin{proof}[Proof of Theorem \ref{T2}]
Part (I)  follows immediately from Corollary \ref{cor1} with $f\equiv 1$.\\
\smallskip
(II)  It is well-known that the elliptic equation
\begin{equation}\label{se}
-\Delta u= u^p+g(x),\quad x\in \mathbb{R}^N,
\end{equation}
where $N\geq 3$ and $p>\frac{N}{N-2}$, admits positive solutions for some $g> 0$ (see e.g.  \cite{BP}). Clearly,  if $u$ is a  positive solution to \eqref{se}, then it is a global positive solution to \eqref{P} with $w=g(x)$ and suitable initial data. This completes the proof of Theorem \ref{T2}.
\end{proof}

Next, we prove Theorem \ref{T3}, which provides the second critical exponent for problem \eqref{P}.

\begin{proof}[Proof of Theorem \ref{T3}]
(I) Let $w=g(x)\in \mathbb{I}_a$. For $T$ large enough, one has
\begin{eqnarray*}
\int_{0< |x|< T^{\frac{q(p-1)}{2p(q-1)}}} g(x)\,dx &\geq & \int_{\frac{1}{2}T^{\frac{q(p-1)}{2p(q-1)}}< |x|< T^{\frac{q(p-1)}{2p(q-1)}}} g(x)\,dx\\
&\geq & C \int_{\frac{1}{2}T^{\frac{q(p-1)}{2p(q-1)}}< |x|< T^{\frac{q(p-1)}{2p(q-1)}}} |x|^{-a}\,dx\\
&\geq & C T^{\frac{q(p-1)(N-a)}{2p(q-1)}}.
\end{eqnarray*}
Hence, for any $0<c_1<c_2<1$, one obtains
\begin{eqnarray*}
T^{\frac{q}{q-1}\left[1-\frac{N(p-1)}{2p}\right]-1} \int_{c_1T}^{c_2T} \int_{0< |x|< T^{\frac{q(p-1)}{2p(q-1)}}} w(t,x)\,dx\,dt&=&  C
T^{\frac{q}{q-1}\left[1-\frac{N(p-1)}{2p}\right]}  \int_{0< |x|< T^{\frac{q(p-1)}{2p(q-1)}}} g(x)\,dx\\
&\geq & CT^{\frac{q}{q-1} \left(1-\frac{a}{a^*}\right)}.
\end{eqnarray*}
Since $a<a^*$, \eqref{gbl} is satisfied, which yields the desired result. \\
\smallskip (II) Let $a^*\leq a<N$. We take
$$
u(x)=\varepsilon (1+|x|^2)^{-\frac{\delta}{2}},\quad  x\in \mathbb{R}^N,
$$
where
\begin{equation}\label{cd}
a-2\leq \delta<N-2 \quad\mbox{and}\quad 0<\varepsilon<[\delta (N-\delta-2)]^{\frac{1}{p-1}}.
\end{equation}
One can show easily that
$$
-\Delta u =u^p +g(x),\quad x\in \mathbb{R}^N,
$$
where
$$
g(x)=\varepsilon \delta \left(N+(N-\delta-2)|x|^2\right) (1+|x|^2)^{-\frac{\delta}{2}-2}-\varepsilon^p (1+|x|^2)^{-\frac{\delta p}{2}}.
$$
Using \eqref{cd}, for all $x\in \mathbb{R}^N$, one obtains
\begin{eqnarray*}
g(x)&\geq & \varepsilon \delta \min\{N,N-\delta-2\} (1+|x|^2)^{-\frac{\delta}{2}-1}-\varepsilon^p (1+|x|^2)^{-\frac{\delta p}{2}}\\
&=& \varepsilon \delta  (N-\delta-2) (1+|x|^2)^{-\frac{\delta}{2}-1}-\varepsilon^p (1+|x|^2)^{-\frac{\delta p}{2}}\\
&=& \varepsilon (1+|x|^2)^{-\frac{\delta}{2}-1}\left( \delta  (N-\delta-2) -\varepsilon^{p-1} (1+|x|^2)^{-\frac{\delta p}{2}+\frac{\delta}{2}+1}\right)\\
&\geq & \varepsilon (1+|x|^2)^{-\frac{\delta}{2}-1} \left(\delta  (N-\delta-2) -\varepsilon^{p-1} \right)\\
&>&0.
\end{eqnarray*}
On the other hand, using \eqref{cd}, for $|x|$ large, one has
\begin{eqnarray*}
g(x)&\leq & \varepsilon \delta \left(N+(N-\delta-2)|x|^2\right) (1+|x|^2)^{-\frac{\delta}{2}-2}\\
&\leq & \varepsilon \delta \max\{N,N-\delta-2\} (1+|x|^2)^{-\frac{\delta}{2}-1}\\
&=& \varepsilon \delta N (1+|x|^2)^{-\frac{\delta}{2}-1}\\
&\leq & \varepsilon \delta N  (1+|x|^2)^{-\frac{a}{2}}\\
&\leq & C |x|^{-a}.
\end{eqnarray*}
Hence, $u$ is a global positive solution to problem \eqref{P} with $w=g(x)\in \mathbb{J}_a$ and suitable initial data. This completes the proof of Theorem \ref{T3}.
\end{proof}

\noindent{\bf Acknowledgements.} The first and third authors extend their appreciation to the Deanship of Scientific Research at King Saud University for funding this work through research group No. RGP--237. The second author is supported by Natural Science Foundation of Zhejiang Province(LY18A010008), Postdoctoral Research Foundation of China(2017M620128, 2018T110332).

\noindent Mohamed Jleli\\
Department of Mathematics, College of Science, King Saud University, P.O. Box 2455, Riyadh, 11451, Saudi Arabia \\
E-mail: jleli@ksu.edu.sa\\

\noindent Ning-An Lai\\
Institute of Nonlinear Analysis and Department of Mathematics, Lishui University, Lishui 323000, China\\
School of Mathematical Sciences, Fudan University, Shanghai 200433, China\\
E-mail: ninganlai@lsu.edu.cn\\

\noindent Bessem Samet\\
Department of Mathematics, College of Science, King Saud University, P.O. Box 2455, Riyadh, 11451, Saudi Arabia \\
E-mail: bsamet@ksu.edu.sa


\begin{thebibliography}{99}

\bibitem{A}
R. Agemi, Blow-up of solutions to nonlinear wave equations in two space dimensions, Manuscripta Math. 73 (1991) 153--162.

\bibitem{BP}
P. Baras, M. Pierre, Crit\`ere d'existence de solutions positives pour des \'equations semi-lin\'eaires non monotones, Ann. Inst. Poincar\'e, Analyse Non Lin\'eaire. 2 (1985) 185--212.


\bibitem{Fujita}
H. Fujita, On the blowing-up of solutions to the Cauchy problem for $u_t = \Delta u+ u^{1+\alpha}$, J. Fac. Sci. Univ. Tokyo Sect. IA Math. 13 (1990) 109--124.

\bibitem{GE}
V. Georgiev, H. Lindblad, C.D. Sogge, Weighted Strichartz estimates and global existence for semilinear wave equations, Amer. J. Math. 119 (6) (1997) 1291--1319.

\bibitem{G1}
R.T. Glassey, Finite-time blow-up for solutions of nonlinear wave equations, Math. Z. 177 (1981) 323--340.

\bibitem{G2}
R.T. Glassey, Existence in the large for $\square u=F(u)$ in two space dimensions, Math. Z. 178 (1981) 233--261.

\bibitem{HZ}
W. Han, Y.  Zhou,  Blow  up  for  some  semilinear  wave  equations  in  multi-space  dimensions,  Comm.  Partial  Differential Equations. 39 (2014) 651--665.


\bibitem{HT}
K. Hidano, K. Tsutaya, Global existence and asymptotic behavior of solutions for nonlinear wave equations, Indiana Univ. Math. J. 44 (1995) 1273--1305.

\bibitem{HWY}
K. Hidano, C. Wang, K. Yokoyama, The Glassey conjecture with radially symmetric data, J. Math. Pures Appl. (9) 98 (2012) 518--541.

\bibitem{HWY2}
K. Hidano, C. Wang, K. Yokoyama, Combined effects of two nonlinearities in lifespan of small solutions to semi-linear wave equations, Math. Ann. 366 (2016) 667--694.





\bibitem{JI}
H. Jiao, Z. Zhou, An elementary proof of the blow up for semilinear wave equation in high space dimensions, J. Differential Equations. 189 (2) (2003) 355--365.


\bibitem{J}
F. John, Blow-up of solutions of nonlinear wave equations in three space dimensions, Manuscripta Math. 28 (1–3) (1979) 235--268.

\bibitem{J2}
F. John, Blow-up for quasilinear wave equations in three space dimensions, Comm. Pure Appl. Math. 34 (1981) 29--51.



\bibitem{LA}
N.A. Lai, Y. Zhou, An elementary proof of Strauss conjecture, J. Funct. Anal. 267 (5) (2014) 1364--1381.

\bibitem{LN}
T.Y. Lee, W.M. Ni, Global existence, large time behavior and life span on solution of a semilinear parabolic Cauchy problem,
Trans. Amer. Math. Soc. 333 (1992) 365--378.

\bibitem{L}
H. Lindblad, C.D. Sogge, Long-time existence for small amplitude semilinear wave equations, Amer. J. Math. 118 (5) (1996) 1047--1135.

\bibitem{MP}
E. Mitidieri, S.I. Pohozaev, A priori estimates and blow-up of solutions of nonlinear partial differential equations and inequalities, Proc. Steklov Inst. Math. 234 (2001) 1--362.

\bibitem{RAA}
M.A. Rammaha, Finite-time blow-up for nonlinear wave equations in high dimensions, Comm. Partial Differential Equations. 12 (1987) 677--700.



\bibitem{RA}
M.A. Rammaha, Differential equations and applications, in: Proceeding of the International Conference on Theory and
Applications of Differential Equations, Vol. I, II, Columbus, OH, March 21--25, 1988, Ohio University Press, Athens, OH, 1989, pp. 322--326.

\bibitem{SC}
J. Schaeffer, The equation $\square u = |u|^p$ for the critical value of $p$, Proc. Roy. Soc. Edinburgh Sect. A 101 (1--2) (1985) 31--44.

\bibitem{SC2}
J. Schaeffer, Finite-time blow-up for $u_{tt}- \Delta u = H(u_r,u_t)$,  Comm. Partial Differential Equations. 11 (1986) 513--543.

\bibitem{SI2}
T.C. Sideris, Global behavior of solutions to nonlinear wave equations in three dimensions, Comm. Partial Differential Equations. 8 (1983) 1291--1323.



\bibitem{SI}
T.C. Sideris, Nonexistence of global solutions to semilinear wave equations in high dimensions, J. Differential Equations. 52 (1984) 378--406.



\bibitem{S}
W.A. Strauss, Nonlinear scattering theory at low energy, J. Funct. Anal. 41 (1981) 110--133.


\bibitem{TW}
  H. Takamura, K. Wakasa, The sharp upper bound of the lifespan of solutions to critical semilinear wave equations in high
dimensions, J. Differential Equations, 251 (2011) 1157--1171.

\bibitem{T}
N. Tzvetkov, Existence of global solutions to nonlinear massless Dirac system and wave equation with small data, Tsukuba J.Math. 22 (1998) 193--211.


\bibitem{Y}
B. Yordanov, Qi S. Zhang, Finite time blow up for critical wave equations in high dimensions, J. Funct. Anal. 231 (2) (2006) 361--374.




\bibitem{Zhang}
Qi S. Zhang, A new critical behavior for nonlinear wave equations, J. Comput. Anal. Appl. 2 (2000) 277--292.



\bibitem{ZH}
Y. Zhou, Cauchy problem for semilinear wave equations with small data in four space dimensions, J. Partial Differential Equations. 8 (2) (1995) 135--144.

\bibitem{ZHH}
Y. Zhou, Blow up of solutions to the Cauchy problem for nonlinear wave equations. Chinese Ann. Math. Ser. B 22 (2001) 275--280.



\bibitem{ZH2}
Y. Zhou, Blow up of solutions to semilinear wave equations with critical exponent in high dimensions, Chinese Ann. Math. Ser. B 28 (2) (2007) 205--212.

\end{thebibliography}
\end{document}